\documentclass[12pt,reqno]{amsart}
\usepackage{tikz,dsfont}
\usepackage{amsmath, amssymb, graphicx, mathrsfs, hyperref}
\usepackage{verbatim} 
\usepackage{enumerate}
\newtheoremstyle{dotless}{}{}{\itshape}{}{\bfseries}{}{ }{} 
\theoremstyle{dotless}

\newtheorem{theorem}{Theorem}[section]
\newtheorem{thm}{Theorem}
\newtheorem{lemma}{Lemma}

\newtheorem{prop}{Proposition}

\newtheorem{cor}{Corollary}

\newcommand{\Mod}[1]{\ (\mathrm{mod}\ #1)}

\begin{document} 

\title{The Frobenius postage stamp problem, and beyond}

\author{Andrew Granville and George Shakan}
\address{AG: D\'epartement de math\'ematiques et de statistique, Universit\'e de Montr\'eal, CP 6128 succ. Centre-Ville, Montr\'al, QC H3C 3J7, Canada; and Department of Mathematics, University College London, Gower Street, London WC1E 6BT, England. }
\email{andrew@dms.umontreal.ca}
\address{GS: Mathematical Institute, University of Oxford, Andrew Wiles Building, Radcliffe Observatory Quarter, Woodstock Road, Oxford, OX2 6GG, UK.}
\email{george.shakan@gmail.com}

\thanks{A.G.~was funded by the European Research Council grant agreement n$^{\text{o}}$ 670239, and by the Natural Sciences and Engineering Research Council of Canada (NSERC) under the Canada Research Chairs program.  G.S.~was supported by Ben Green's Simons Investigator Grant 376201. Many thanks to Seva Lev and Tyrrell McAllister for pointing us to the references \cite{Sav} and  \cite{ST}, respectively.}

\dedicatory{Dedicated to Endre Szemer\'edi on the occasion of his 80th birthday}


\date{}

\begin{abstract}
Let $A$ be a finite subset of $\mathbb{Z}^n$, which generates $\mathbb{Z}^n$ additively. We provide a precise description of the $N$-fold sumsets $NA$ for $N$ sufficiently large, with some explicit bounds on ``sufficiently large."
\end{abstract}
\maketitle

\section{Introduction}

Let $A$ be a given finite subset of the integers.  For any integer $N \geq 1$, we are interested in determining the $N$-fold sumset of $A$,
\[
NA : = \{a_1 + \cdots + a_N : a_1 , \ldots , a_N \in A\} ,
\]
where the $a_i$'s are not necessarily distinct.
For simplicity we may assume without loss of generality that the smallest element of $A$ is $0$, and that the gcd of its elements is $1$.\footnote{Since if we translate $A$ then we translate $NA$ predictably, as $N(A+\tau)=NA+N\tau$, and 
since if $A=g\cdot B:=\{ gb: b\in B\}$ then $NA=g\cdot NB$.} Under these assumptions we know  that 
\[
0\in A\subset 2A\subset 3A\subset \cdots \subset \mathbb N,
\]
where $\mathbb N$ is the natural numbers, defined to be the integers $\geq 0$.
Moreover there exist integers $m_1,\ldots,m_k$ such that $m_1 a_1 + \cdots + m_k a_k = 1$, and therefore
\[
\mathcal P(A)   = \left\{ \sum_{a\in A} n_aa:\ \text{ Each } n_a\in \mathbb{N}\right\}
=\lim_{N\to \infty} NA = \mathbb N \setminus \mathcal E(A)
\]
for some finite \emph{exceptional set} $\mathcal E(A)$.\footnote{We give a simple proof that $\mathcal E(A)$ is finite in 
section \ref{sec: basics}.}

One very special case is the \emph{Frobenius postage stamp problem} in which we wish to determine what exact postage cost one can make up from an unlimited of $a$ cent and $b$ cent stamps. In other words, we wish to determine
$\mathcal P(A)$ for $A=\{ 0,a,b\}$. It is a fun challenge for a primary school student to show that 
$\# \mathcal E(\{ 0,3,5\})=\{ 1,2,4,7\}$, and more generally, \cite{Syl}, that  
\[
\max \mathcal E(\{ 0,a,b\}) = ab - a - b, \text{ and }  |\mathcal E(\{ 0,a,b\})| = \tfrac12 (a-1)(b-1).
\]
Erd\H{o}s and Graham \cite{EG} conjectured precise bounds for  $\max \mathcal E(A)$; see also Dixmier \cite{Di}.

In this article we study the variant in which we only allow the use of at most $N$ stamps; that is, can we determine the structure of the set $NA$?   If $b = \max A$, then $NA\subset \{0 , \ldots , bN\} \cap \mathcal P(A)= \{0 , \ldots , bN\} \setminus \mathcal E(A)$.  Moreover, we can  use symmetry to determine a complementary exceptional set:\   Define the set $b-A:=\{ b-a:\ a\in A\}$. Then $NA=Nb-N(b-A)$ and so $NA$ cannot contain any elements
$Nb-e$ where $e\in \mathcal E(b-A)$. Therefore
\[
NA \subset \{0 , \ldots , bN\} \setminus (\mathcal E(A) \cup (bN - \mathcal E(b-A))).
\]
We ask when equality holds?

\begin{thm} \label{Thm1} 
Let $A$ be a given finite subset of the integers, with smallest element $0$ and largest element $b$, in which the gcd of the elements of $A$ is $1$.
If  $N\geq 2[\tfrac b2]$ and  $0 \leq n\leq Nb$ with $n\not\in \mathcal E(A)\cup (Nb-  \mathcal E(b-A))$ then $n\in NA$. Equivalently, we have 
$$NA = \{0 , \ldots , bN\} \setminus (\mathcal E(A) \cup (bN - \mathcal E(b-A))).$$
\end{thm}

In the next section we will show that if $A$ has just three elements then Theorem \ref{Thm1} holds for all integers $N\geq 1$ (which does not seem to have been observed before). However this is not true for larger $A$:  If  $A=\{ 0,1,b-1,b\}$ then $\mathcal{E}(A) = \mathcal{E}(b-A)=\emptyset$ and  $b-2\in (b-2)A$ but $b-2\notin (b-3)A$, in which case Theorem \ref{Thm1} can only hold for $N\geq b-2$. We conjecture that one should be able to obtain the lower bound ``$N\geq b-2$'' (which would then be best possible) in place of ``$N\geq 2[\tfrac b2]$'' in Theorem \ref{Thm1}.\footnote{Bearing in mind the example $A=\{ 0,1,N+1,N+2,\dots,b\}$, we can refine this conjecture to ``$N\geq b+2-\#A$'' whenever $\#A\geq 4$.} It is feasible that one could develop our methods to show this, but it seems to us like that would be a formidable task. 

Theorem \ref{Thm1} seems to have first been proved, but with the bound $N\geq b^2(\# A-1)$,  by Nathanson \cite{Nath} in 1972, which was improved to  $N\geq \sum_{a\in A,\ a\ne 0} (a-1)$ in \cite{WCC}.\footnote{\cite{WCC} claim that their result is ``best possible," but this is a consequence of how they formulate their result. Indeed Theorem \ref{Thm1}  yields at least as good a bound \emph{for all}  sets $A$ with $\# A\geq 4$, and is better in all but a couple of families of examples.}

We will generalize Theorem \ref{Thm1} to sets $A$ of arbitrary dimensions. Here we assume that $0\in A \subset \mathbb{Z}^n$. 
The \emph{convex hull} of the points in $A$ is given  by 
 $$
 H(A) = \left\{\sum_{a \in A} c_a a : \sum_{a\in A} c_a = 1, \text{ each } c_a\geq 0\right\},
 $$
so that 
\[
C_A := \left\{\sum_{a \in A} c_a a :  \text{Each } c_a \geq 0\right\} = \lim_{N\to \infty} N H(A),
\]
is the \emph{cone} generated by $A$. Let $\mathcal P (A) $ be the set of sums in $C_A$ where each $c_a\in \mathbb N$, so that $\mathcal P (A)\subset C_A \cap\mathbb{Z}^n$. We define the \emph{exceptional set} to be
 \[
 \mathcal E (A) : =  
 (C_A \cap \mathbb{Z}^n) \setminus \mathcal P (A),
\]
 the   integer points that are in the convex hull of positive linear combinations of points from $A$, and yet are not an element of   $NA$, for any integer $N\geq 1$.
With this notation we can formulate our result:

\begin{thm} \label{Thm2} 
Let $0\in A \subset \mathbb{Z}^n$ such that $A$ spans $\mathbb Z^n$ as a vector space over $\mathbb Z$.  There exists a constant $N_{A}$ such that if $N\geq N_{A}$,
\[
NA = (NH(A)\cap \mathbb Z^n) \setminus  \mathcal E_N(A) \text{ where } \mathcal E_N(A):= \bigg( \mathcal E(A)\cup \bigcup_{a\in A} (aN-\mathcal E(a-A) ) \bigg) .
\]
\end{thm}


We have been unable to find exactly this result in the literature. 
It would be good to obtain an upper bound on $N_A$, presumably in terms of the geometry of the convex hull of $A$.

In Theorem \ref{Thm1}, when $A\subset \mathbb N^1$,  the sets $\mathcal E(A)$ are finite, which can be viewed as a finite union of $0$ dimensional objects. In the two dimensional example
\begin{equation}\label{ex1} A = \{(0,0) , (2,0) , (0,3) , (1,1)\},\end{equation}
we find that $\mathcal{E} (A)$ in infinite, explicitly
\begin{align*}
\mathcal{E} (A)&=\quad  \{ (0,1), (1,0), (1,2)  \} + \mathcal P(\{ (0,0),(2,0)\})  \\
&  \cup \quad \{ (0,1), (0,2), (1,0), (1,2), (2,1), (3,0)\} + \mathcal P(\{ (0,0),(0,3)\}),
\end{align*}
the  union of nine one-dimensional objects. More generally we prove the following:

\begin{thm} \label{Thm3} 
Let $0\in A \subset \mathbb{Z}^n$ such that $A$ spans $\mathbb Z^n$ as a vector space over $\mathbb Z$.  Then 
$\mathcal E(A)$ is a finite union of sets of the form
\[
\bigg\{  v+\sum_{b\in B} m_b b: m_b\in \mathbb Z_{\geq 0}\bigg\} = v+\mathcal P(B\cup\{ 0\})
\]
where $v\in C_A\cap \mathbb Z_{\geq 0}^n$, with $B\subset A$ contains $\leq n-1$ elements, and the vectors in $B-0$ are linearly independent.
\end{thm}

We deduce from Theorem \ref{Thm3}  that
  \begin{equation}\label{E} 
  \# \mathcal{E}_N (A)  = O(N^{n-1}).
  \end{equation}
Theorem \ref{Thm3} also implies that there is a bound $B_A$ such that 
every element of $C_A \cap \mathbb{Z}^n$ which is further than a distance $B_A$ from its boundary, is an element of 
$\mathcal{P} (A)$ (and so not in $\mathcal{E} (A)$).  

The most remarkable result in this area is the 1992 theorem of  Khovanskii \cite[Corollary 1]{Khov} who proved that $\# NA$ is a polynomial of degree $n$ in $N$ for $N$ sufficiently large, where the leading coefficient is $ {\rm Vol}(H(A))$.
His extraordinary proof proceeds by constructing a finitely-generated graded module $M_1,M_2,\ldots$ over  $\mathbb C[t_1,\dots,t_k]$ with  $k=\#A$, where each $M_N$ is a vector space over $\mathbb C$ of dimension $|NA|$. One then deduces that $|NA|=\text{dim}_{\mathbb C} M_N$ is a polynomial in $N$, for $N$ sufficiently large, by a theorem of Hilbert. 
Nathanson \cite{Nath2} showed that this can generalized to sums $N_1A_1+\cdots+N_kA_K$ when all the $N_i$ are sufficiently large. This was all reproved by Nathanson and Ruzsa \cite{NathRuz} using elementary, combinatorial ideas (using several ideas in common with us).  Moreover it can also be deduced from Theorems  \ref{Thm2} and  \ref{Thm3} .

In section~\ref{easy} we look at the case where $A$ has three elements, showing that  the result holds for all $N\geq 1$. This easier case introduces some of the ideas we will need later.  In section~\ref{notsoeasy} we prove Theorem \ref{Thm1}. Obtaining the bound $N\geq 2b-2$ is not especially difficult, but improving this to $N\geq 2[\tfrac b2]$ becomes complicated and so we build up to it in a number of steps. In section~\ref{high} we begin the study of a natural higher dimensional analog. The introduction of even one new dimension creates significant complications, as the exceptional set $\mathcal{E} (A)$ is no longer necessarily finite.
In the next subsection we indicate how one begins to attack these questions.

\subsection{Representing most elements of $\mathbb Z_{\geq 0}^n$} \label{sec: basics}
If  $A=\{ 0,3,5\}$ one can represent
\[
8=1\times 3+1\times 5, \quad 9 =3\times 3\ \text{ and }\ 10=2\times 5
\]
and then every integer $n\geq 11$ is represented by adding a positive multiple of 3 to one of these representations, depending on whether
$n\equiv 2,0$ or $1 \mod 3$, respectively. In effect we are find representatives $r_1=10, r_2=8, r_3=9$ of $\mathbb Z/3\mathbb Z$ that belong to $\mathcal P(A)$, and then $\mathbb Z_{\geq 8} =\{ r_1,r_2,r_3\}+3\mathbb Z_{\geq 0} \subset \mathcal P(A)$, which implies that 
$\mathcal E(A)\subset \{ 0,\dots,7\}$. 

We can generalize this to arbitrary finite $A\subset \mathbb Z_{\geq 0}$ with $\text{gcd}(a:\ a\in A)=1$, as follows: Let $b\geq 1$ be the largest element of $A$ (with $0$ the smallest).  Since $\text{gcd}(a:\ a\in A)=1$ there exist integers $m_a$, some positive, some negative, for which $\sum_{a\in A} m_aa=1$. Let $m:=\max_{a\in A} (-m_a)$ and $N:=bm\sum_{a\in A} a$, so that 
\[
r_k:=N+k= \sum_{a\in A} (bm+km_a)a\in \mathcal P(A) \text{ for } 1\leq k\leq b
\]
 (as each $bm+km_a\geq bm-km\geq 0$) and 
$r_k\equiv k \mod b$. But then 
\[
\mathbb Z_{> N}=N+ b\mathbb Z_{\geq 1}  =N+\{ 1,\dots, b\} +b\mathbb Z_{\geq 0}= \{ r_1,\dots, r_{b}\}+b\mathbb Z_{\geq 0} \subset \mathcal P(A),
\]
 which implies that  $\mathcal E(A)\subset \{ 0,\dots,N\}$. 

We can proceed similarly in $\mathbb Z_{\geq 0}^n$ with $n>1$, most easily when $C_A$ is generated by a set $B$ containing exactly $n$ non-zero elements (for example,  $B:=\{(0,0), (2,0), (0,3)\}\subset A$,  in the example from \eqref{ex1}).
Let $\Lambda_B$ be the lattice of   integer linear combinations of elements of $B$. We need to find $R\subset \mathcal P(A)$, a set of representatives of $\mathbb Z^n/\Lambda_B$, and then $(R+C_B)\cap \mathbb Z^n\subset \mathcal P(A)$.  In the example \eqref{ex1} we can easily represent $\{ (m,n)\in \mathbb Z^2: 4\leq m\le 5,\ 3\leq n\leq 5\}$. Therefore if
$(r,s)\in \mathcal E(A)$ then either $0\leq r\leq 3$ or $0\leq s\leq 2$, and so we see that $\mathcal E(A)$ is a subset of a finite set of translates of one-dimensional objects.

\section{Classical postage stamp problem with at most $N$ stamps}\label{easy}
It is worth pointing out explicitly that if,  for given coprime integers $0 < a < b$, we have $n\in N\{0,a,b\}$ so that
$n = ax + by$ with $x + y \leq N$ then\footnote{In this displayed equation, and throughout, we write ``$r \times a$'' to mean $r$ copies of the integer $a$.}
\[ (N - x - y)\times b + x \times (b-a)   = bN - n\]
so that  $bN - n \in N\{ 0,b-a,b\}$.  
 
\begin{thm}[Postage Stamp with at most $N$ stamps] \label{Thm0}
Let $0 < a < b$ be coprime integers and $A = \{0,a,b\}$. If $N\geq 1$ then 
\[ NA = \{0 , \ldots , bN \} \setminus (\mathcal E(A) \cup (bN - \mathcal E(b-A))).\]
\end{thm}

In other words, $NA$ contains all the integers in  $[0,bN]$, except  a few unavoidable exceptions near to the endpoints of the interval. 

\begin{proof}
Suppose that $n \in \{0 , \ldots , bN\}$, $n \notin \mathcal E(A)$ and $bN-n \notin   \mathcal E(b-A)$, so that
there exist $r,s ,r',s'\in \mathbb N$ such that 
\begin{equation}\label{E1} 
ra + sb = n,
\end{equation} 
and 
\begin{equation}\label{E2} 
r' (b-a) + s' b = bN - n.
\end{equation}
We may assume $0 \leq r,r' \leq b-1$, as we may replace $r$ with $r- b$ and $s$ with $s + a$, and
$r'$ with $r'- b$ and $s'$ with $s'+ b-a$. Now reducing \eqref{E1} and \eqref{E2} modulo $b$, we have
\[
ra \equiv n \Mod b , \ \ \ -r' a \equiv -n \Mod b.
\]
Since $(a,b) = 1$, we deduce $r\equiv r' \Mod b$. Therefore $r=r'$ as $|r-r'|<b$, and so adding \eqref{E1} and \eqref{E2} we find 
\[
r b + s b + s' b = bN.
\]
This implies that $r+s+s'=N$ and so $r+s\leq N$ which gives $n\in NA$, as desired.
\end{proof}

\section{Arbitrary postage problem with at most $N$ stamps}\label{notsoeasy}

\subsection{Sets with three or more   elements} Let 
$$A=\{ 0=a_1<a_2<\ldots <a_k=b\} \subset \mathbb{Z},$$ 
with $(a_1,\ldots,a_k)=1$. In general we have $n\in NA$ if and only if $Nb-n\in N(b-A)$, since 
\[
n = \sum_{i=1}^k m_i a_i \text{ if and only if } 
Nb-n = \sum_{i=1}^k m_i (b- a_i)
\]
where we select $m_1$ so that $\sum_{i=1}^k m_i=N$.
For $0\leq a\leq b-1$  define
\[
n_{a,A}:=\min\{ n\geq 0:\ n\equiv a \Mod b \text{ and } n\in \mathcal P(A)\} 
\]
and 
\[
N_{a,A}:=\min\{ N\geq 0:\ n_{a,A}\in NA\} 
\]
We always have $n_{0,A}=0$ and $N_{0,A}=0$. Neither $0$ nor $b$ can be a term in the sum for $n_{a,A}$ else we can remove it and contradict the definition of $n_{a,A}$. But this implies that $n_{a,A}\leq N_{a,A} \cdot \max_{c\in A: c<b} c \leq (b-1)N_{a,A}$.

\begin{lemma} \label{lem1}
If $n\equiv a \Mod b$  then $n\in \mathcal P(A)$ if and only if $n\geq n_{a,A}$.
\end{lemma} 

\begin{proof} If $n<n_{a,A}$ then $n\not \in \mathcal P(A)$ by the definition of $n_{a,A}$.
Write $n_{a,A}=\sum_{c\in A} n_cc$ where each $n_c\geq 0$. If $n\equiv a \Mod b$  and $n\geq n_{a,A}$ then $n=n_{a,A}+rb$ for some integer $r\geq 0$ and so $n=\sum_{c\in A, c\ne b} n_cc
(n_b+r)b\in \mathcal P(A)$.
\end{proof}

We deduce that 
\[
 \mathcal E(A)=\bigcup_{a=1}^{b-1}\ \{ 1\leq n< n_{a,A}:\  n\equiv a \Mod b\} ;
\]
We also have the following:

\begin{cor} \label{cortolem1} Suppose that $0\leq n\leq bN$ and $n\equiv a \pmod b$.
Then 
\[ n\not\in \mathcal E(A)\cup (Nb-  \mathcal E(b-A)) \text{  if and only if  } n_{a,A}\leq n\leq bN-n_{b-a,b-A} .\]  Thus there are such integers $n$ if and only if 
 $N\geq N_{a,A}^*:= \tfrac 1b (n_{a,A}+n_{b-a,b-A})$.
\end{cor}

\begin{lemma}\label{Lem: Indn}
Suppose that $N_0\geq N_{a,A}^*$. 
Assume that if $0\leq n\leq bN_0$ with $n\equiv a \pmod b$, and
$n\not\in \mathcal E(A)\cup (N_0b-  \mathcal E(b-A))$ then $n\in N_0A$. Then for any integer $N\geq N_0$ we have 
$n\in NA$ whenever  $0\leq n\leq bN$ with $n\equiv a \pmod b$, and
$n\not\in \mathcal E(A)\cup (Nb-  \mathcal E(b-A))$.
\end{lemma}

\begin{proof} By induction. By hypothesis it holds for $N=N_0$. Suppose it holds for some $N\geq N_0$.
If $n\equiv a \pmod b$ with $a\leq n\leq  b(N+1)-n_{b-a,b-A}$ then either $a\leq n\leq  bN-n_{b-a,b-A}$ so that $n\in NA\subset (N+1)A$, or 
 $n=b+(bN-n_{b-a,b-A})\in b+NA\subset (N+1)A$. 
\end{proof}

 If $n_{a,A}=a_1+\cdots+a_N$ where $N=N_{a,A}$ then 
 $$bN_{a,A}-n_{a,A}=(b-a_1)+\cdots+(b-a_N)\geq n_{b-a,b-A},$$
by definition. Therefore
 \[
 N_{a,A}  \geq \tfrac 1b (n_{a,A}+n_{b-a,b-A})=N_{a,A}^*,
 \]
 and the analogous argument implies that $N_{b-a,b-A} \geq  N_{a,A}^*$.

\begin{cor} \label{cor2tolem1} 
Given a set $A$, fix  $a \Mod b$. The statement ``For all integers $N\geq 1$, for all integers $n\in [0,Nb]$ with $n\equiv a \Mod b$ we have  $n\in NA$ if and only if  $n\not\in \mathcal E(A)\cup (Nb-  \mathcal E(b-A))$'' holds true if and only if $N _{a,A}=N_{a,A}^*$.
 \end{cor}

\begin{proof}  There are no such integers $n$ if $N< N_{a,A}^*$ by Corollary  \ref{cortolem1}, so the statement is true. If the 
statement is true for  $N= N_{a,A}^*$ then it holds for all $n\geq N_{a,A}^*$ by Lemma \ref{Lem: Indn}. Finally for $N= N_{a,A}^*$, the statement claims (only) that  $n_{a,A}\in NA$. This happens if and only if $N=N_{a,A}^*\geq N _{a,A}$.  The result follows since we just proved that $N_{a,A}  \geq  N_{a,A}^*$. 
\end{proof}

In fact one can re-run the proof on $bN-a$ to see that if $N _{a,A}=N_{a,A}^*$ then 
$N _{b-a,b-A}=N_{a,A}^*$.  
Suppose $A$ has just three elements, say $A=\{0,c,b\}$ with $(c,b)=1$.
 For any non-zero $a \Mod b$ we have an integer $r, 1\leq r\leq b-1$ with $a\equiv cr \Mod b$, and one can easily show that 
$n_{a,A}=cr$ while $N_{a,A}=r$. Now $b-A=\{0,b-c,b\}$ so that 
$n_{b-a,b-A}=(b-c)r$ while $N_{b-a,b-A}=r$. Therefore
$N_{a,A}=N_{b-a,b-A}=N_{a,A}^*=\tfrac 1b(n_{a,A}+n_{b-a,b-A})$ for every $a$, and so we recover Theorem \ref{Thm0} from Corollary \ref{cor2tolem1}.

However Theorem \ref{Thm1} does not hold for all $N\geq 1$ for some sets $A$ of size $4$.
For example, if $A=\{ 0,1,b-1,b\}$ then $b-A=A$. We have $n_{a,A}=a$ for $1\leq a\leq b-1$, and so $N_{a,A}^*=1$, but  $N_{a,A}=a$ for $1\leq a\leq b-2$, and so Theorem \ref{Thm1} does not hold for all $N\geq 1$ by Corollary \ref{cor2tolem1}. In fact since $N_{b-2,b}=b-2>N_{b-2,b}^*=1$, if the statement ``if $n\leq Nb$ and $n\not\in \mathcal E(A)\cup (Nb-  \mathcal E(b-A))$ then $n\in NA$'' is true then $N\geq b-2$.

It would be interesting to have a simple criterion for the set $A$ to have the property that $N _{a,A}=N_{a,A}^*$ for all $a \Mod b$
(so that Corollary \ref{cor2tolem1} takes effect).
Certainly many sets $A$ do not have this property; For example if there exists an integer $a,\ 1\leq a\leq b-1$ such that 
$a\not\in A$ but $a, b+a\in 2A$, then $n_{a,A}=a,\ n_{b-a,b-A}=b-a$, so that $N _{a,A}=2$ and $N_{a,A}^*=1$.

\subsection{Proving a ``sufficiently large'' result}
 
 We begin getting bounds by proving the following.
 
 \begin{prop} \label{keyprop}
Fix $0\leq a\leq b-1$ and suppose $N\geq N_{a,A}+N_{b-a,b-A}$.
 If $0 \leq n\leq Nb$ with $n\equiv a \Mod b$ and $n\not\in \mathcal E(A)\cup (Nb-  \mathcal E(b-A))$ then $n\in NA$.  
 \end{prop}

\begin{cor} \label{Cor2} 
If $0 \leq n\leq Nb$ and $n\not\in \mathcal E(A)\cup (Nb-  \mathcal E(b-A))$ then $n\in NA$, 
whenever $N\geq \max_{1\leq a\leq b-1} N_{a,A}+N_{b-a,b-A}$.
\end{cor}

To prove Proposition \ref{keyprop}, we need the following.

\begin{prop} \label{P1} Fix $1\leq a\leq b-1$.
If $n\leq (N-N_{a,A})b$ with $n\equiv a \Mod b$ and $n\not\in \mathcal E(A)$ then $n\in NA$.
\end{prop}


\begin{proof} If $n\not\in \mathcal E(A)$ then  $n\geq n_{a,A}$ by the definition of $n_{a,A}$.
Therefore $n=n_{a,A}+kb$ where $0\leq kb\leq n \leq (N-N_{a,A})b$, so that  
$0\leq k\leq N-N_{a,A}$ and $kb\in (N-N_{a,A})A$. Now
$n_{a,A}\in N_{a,A}A$ and so
$n=n_{a,A}+ kb \in N_{a,A}A + (N-N_{a,A})A=NA$.
\end{proof}

\begin{proof} [Proof of Proposition \ref{keyprop}] This is trivial for $a=0$. Otherwise, by hypothesis $n\not\in \mathcal E(A)$ and $bN-n\not\in \mathcal E(b-A)$.
Moreover either $n\leq (N-N_{a,A})b$ or $bN-n\leq (N-N_{b-a,b-A})b$,   else 
\[
bN=n+(bN-n)>(N-N_{a,A})b+(N-N_{b-a,b-A})b=(2N-N_{a,A}A-N_{b-a,b-A})b\geq Nb,
\]
which is impossible. Therefore Proposition \ref{keyprop} either follows by applying Proposition \ref{P1} to $A$, or by applying Proposition \ref{P1} to $b-A$ to obtain $Nb-n\in N(b-A)$ which implies $n\in NA$.
\end{proof}



It remains to bound $N_{a,A}$. We start with the following.

\begin{lemma}\label{lem1} 
 We have $N_{a,A}\leq b-1$.
If $A=\{ 0,1,b\}$ then $N_{a,A}=b-1$.
\end{lemma}

\begin{proof}
Suppose that $n_{a,A}=a_1+a_2+\cdots+a_r$ with each $a_i\in A$,  and $r$ minimal. We have $r< b$ else two of
$0,a_1,a_1+a_2,\ldots,a_1+\cdots +a_b$ are congruent mod $b$ by the pigeonhole principle, so their difference, which is a subsum of the $a_i$'s is $\equiv 0 \Mod b$. If these $a_i$'s are removed from the sum then we obtain a smaller element of $ \mathcal P(A)$ that is $\equiv a \Mod b$, contradicting the definition of $n_{a,A}$.  We deduce that $N_A\leq b-1$. 
If $A= \{0,1,b\}$ then $b-1 \notin (b-2)A$ and so $N_A\geq b-1$.
\end{proof}

\begin{cor} \label{cor1} 
Suppose that $N\geq 2b-2$.
If $n\leq Nb$ and $n\not\in \mathcal E(A)\cup (Nb-  \mathcal E(b-A))$ then $n\in NA$.
\end{cor}

\begin{proof} Insert the bounds $N_{a,A}, N_{b-a,b-A}\leq b-1$ from Lemma \ref{lem1} into Corollary \ref{Cor2}.
\end{proof}

 \subsection{The proof of Theorem \ref{Thm1}}
 
 With more effort we now prove Theorem \ref{Thm1}, improving upon Corollary~\ref{cor1} by a factor of 2, and getting close to the best possible bound $b-2$ (which, as we have seen, is as good as can be attained when $A=\{ 0,1,b-1,b\}$).
 One cannot obtain a better consequence of  Corollary \ref{Cor2}   since we have the following examples:
\smallskip

If $A=\{ 0, 1, b-1, b\}$ then    $N_{[\tfrac b2],A}+N_{b-[\tfrac b2],b-A}=2[\tfrac b2]$.
 
If $A=\{ 0, 1, 2, b\}$  with $b$ even then    $N_{b-1,A}+N_{1,b-A}=b$.  This is a particularly interesting case as one can verify that one has ``If $n\leq Nb$ and $n\not\in \mathcal E(A)\cup (Nb-  \mathcal E(b-A))$ then $n\in NA$'' \emph{for all} $N\geq 1$.
 
\smallskip

We can apply Corollary \ref{Cor2} to obtain Theorem \ref{Thm1} provided $N_{a,A},N_{b-a,b-A}\leq [\tfrac b2]$ for each $a$. Therefore we need to classify those $A$ for which $N_{a,A}> \tfrac b2$

Let $(t)_b$ is the least non-negative residue of $t \Mod b$. 

Suppose that $1\leq a\leq b-1$, and write $n_a=n_{a,A}=  a_1+\cdots+a_{m}$ where $m=N_{a,A}$ is minimal. No subsum of $a_1+\cdots+a_{m}$ can sum to $\equiv 0 \Mod b$  else we remove this subsum from the sum to   get a smaller sum of  elements of $A$ which is $\equiv a \Mod b$, contradicting the definition of $n_a$.  Also the complete sum cannot be $\equiv 0 \Mod b$ else $a=0$ and $m=0$.  Let $k=m+1$ and $a_k= -(a_1+\cdots+a_{m})$, so that 
$a_1+\cdots+a_{k}\equiv 0 \Mod b$ and no proper subsum is $0 \pmod b$; we call this a 
\emph{minimal zero-sum}. The  Savchev-Chen structure theorem \cite{Sav} states that if $k\geq [\tfrac b2]+2$ then
$a_1+\cdots+a_{k}\equiv 0 \Mod b$ is a minimal zero-sum if and only if there is a reduced residue $w \Mod b$ and  positive integers
$c_1,\dots,c_k$ such that $\sum_j c_j=b$ and  
$a_j\equiv wc_j \Mod b$ for all $j$.

\begin{theorem}\label{SC1}
If $N_{a,A}> \tfrac b2$ then $n_{a,A}$ is the sum of  $N_{a,A}$ copies of some integer $h, 1\leq h\leq b-1$ with $(h,b)=1$.  Moreover if $k\in A$ with $\ell\ne h$  then $(k/h)_b\geq  N_{a,A}+1$. 
\end{theorem}

\begin{proof}
Above we have $k=m+1=N_{a,A}+1\geq [\tfrac b2]+2$  so we can apply the Savchev-Chen structure theorem. Some $c_j$ with $j\leq m$ must equal $1$ else $b=\sum_{j=1}^m c_j\geq 2m>b$, a contradiction.
Hence $h\in A$ where $h=(w)_b$. Let $n:=\#\{ j\in [1,m]: c_j=1\}=\#\{ j\in [1,m]: a_j=h\}\geq 1$.

If $(\ell h)_b\in A$ where $1\leq \ell <b$ then $n\leq \ell -1$ else  we can remove $\ell $
copies of $h$ from the original sum for $n_{a,A}$ and replace them by one copy of $(\ell h)_b$. If $(\ell h)_b<\ell h$ then this makes the sum smaller, contradicting the definition of $n_a$. Otherwise this makes the number of summands smaller contradicting the definition of $N_{a,A}$.
 
Therefore if $k$ is the smallest $c_j$-value $>1$, with $1\leq j\leq m$, then $(k h)_b\in A$  so that
$k\geq n+1$,
and so
\[
b-1\geq \sum_{j=1}^m c_j \geq n\times 1+(m-n)\times k=m+(m-n)(k-1)\geq m+(m-n)n.
\]
If $1\leq n\leq m-1$ then this gives $b-1\geq m+(m-1)>b-1$, a contradiction. Hence $n=m$; that is, $n_a=h+h+\cdots+h$. Therefore $hm\equiv a \Mod b$. Moreover if $(\ell h)_b\in A$ with $\ell \ne 1$ then $\ell \geq n+1= m+1$. 
\end{proof}

We now  give a more precise version of the argument in Proposition \ref{P1}.  

 \begin{prop} \label{P3}
Fix $0\leq a\leq b-1$ and suppose $N\geq \max\{ N_{a,A}, N_{b-a,b-A}\}$.
For all $0 \leq n\leq Nb$ with $n\equiv a \Mod b$ and $n\not\in \mathcal E(A)\cup (Nb-  \mathcal E(b-A))$ we have that $n\in NA$,
except perhaps if $ n=n_{a,A}+jb$ where 
\begin{equation} \label{eq: missingrange}
N-N_{a,A}<j<N_{b-a,b-A}-\frac 1b( n_{a,A}+n_{b-a,b-A}).
\end{equation}
 \end{prop} 

 \begin{proof} Since $n_{a,A}\in N_{a,A}A$, we have 
 \[
 n_{a,A}+jb \in (N_{a,A}+j)A\in NA \text{ whenever } 0\leq j\leq N-N_{a,A}.
 \]
 The analogous statement for $b-A$ implies that 
 \[
 bN_{b-a,b-A}-n_{b-a,b-A}+ib \in NA  \text{ whenever } 0\leq i\leq N-N_{b-a,b-A}.\qedhere
 \]
 \end{proof}

 \begin{proof}  [Proof of Theorem \ref{Thm1}]
Suppose that $N\geq N_0:=2[\tfrac b2]\geq b-1$. We will prove the result now for $N=N_0$; the result for all $N\geq N_0$ follows from Lemma \ref{Lem: Indn}.

If $N_{a,A},N_{b-a,b-A}\leq [\tfrac b2]$ then the result follows from  Proposition \ref{keyprop}.
Hence we may assume that $N_{a,A}> [\tfrac b2]$ (if necessary changing $A$ for $b-A$).

Theorem \ref{SC1} implies there exists an integer $h, 1\leq h\leq b-1$ with $(h,b)=1$ such that 
$n_{a,A}=N_{a,A}\times h$. We already proved the result when $A$ has three elements, so we may now assume it has a fourth, say $\{ 0,h,\ell,b\}\subset A$.

Let $\mathcal B=\{ 0,h,\ell\}\subset \mathbb Z/b\mathbb Z$.  Since $\mathcal B$ is not contained in any proper subgroup of 
$\mathbb Z/b\mathbb Z$ (as $(h,b)=1$), Kneser's theorem implies that $|k\mathcal B|\geq 2k+1$. 

For $N_0-N_{a,A}\leq k\leq \tfrac{b-1}2$,   let $S:=2k-b+N_{a,A}+1$ so that  there are $ b-2k$ elements in $\{ Sh, (S+1)h,\dots, N_{a,A}h\}$. By the pigeonhole principle,  $sh\in k\mathcal B$ for some $s, S\leq s\leq N_{a,A}$ and therefore $sh+tb=a_1+\dots+a_k$ where each $a_i\in A$, for some integer $t$. Now $t\geq 0$ else we can replace $sh$ by $a_1+\dots+a_k$ contradicting the definition of $n_{a,A}$. On the other hand, $tb<sh+tb=a_1+\dots+a_k\leq k(b-1)$ and so  $t\leq k$.  Therefore
\begin{align*}
n_{a,A}+kb&= (N_{a,A}-s)h +(a_1+\dots+a_k) +(k-t)b \in (N_{a,A}-s+2k-t)A\\
& \subset  (N_{a,A}-S+2k-t)A = (b-1-t)A\subset  N_0A.
\end{align*}
We have filled in the range \eqref{eq: missingrange} for all $j\leq \tfrac{b-1}2$, which gives the whole of \eqref{eq: missingrange} if
$N_{b-a,b-A}\leq [\tfrac b2]$. Therefore we may now assume that $N_{b-a,b-A}> [\tfrac b2]$.

Since $N_{b-a,b-A}> [\tfrac b2]$ we may now rerun the argument above and obtain that 
\[
n_{b-a,b-A}+kb\in N_0(b-A)  \text{ for all } k\leq \frac{b-1}2,
\]
and therefore if $n_{a,A}+jb\not\in \mathcal E(b-A)$ then
\[
n_{a,A}+jb\in N_0A  \text{ for all } j\geq  \frac{b-1}2 ,
\]
since 
\[
N_0-\frac{b-1}2-\frac {n_{a,A} +n_{b-a,b-A}}b \leq b-\frac{b-1}2-1=\frac{b-1}2. \qedhere
\]
 \end{proof}

 \section{Higher dimensional postage stamp problem}\label{high}
 
 Let $A = \{a_1 , \ldots , a_k\} \subset \mathbb{Z}^n$ be a finite set of vectors with $k \geq n+2$. After translating $A$, we assume that $0\in A$ so that 
 $$0\in A \subset 2A \subset \cdots.$$
We are interested in what elements are in $NA$.
  Assume that
 \[
 \Lambda_A := \langle A \rangle_{\mathbb{Z}} =  \mathbb{Z}^n.
 \]

 It is evident from the definitions that 
\[
NA \subset NH(A)\cap \mathcal P (A) = (NH(A)\cap \mathbb Z^n) \setminus \mathcal E (A) 
\]
Let $b\in A$ and suppose that $x\in NA$ so that $x=\sum_{a \in A} c_a a$ where the $c_a$ are non-negative integers that sum to $N$. Therefore 
$Nb-x=Nb-\sum_{a \in A} c_a a=\sum_{a \in A} c_a(b-a)\in N(b-A)\subset  \mathcal P (b-A)$.
This implies that $Nb-x\not\in   \mathcal E (b-A)$, and so $x\not\in Nb- \mathcal E (b-A)$.
Therefore 
\[
NA \subset (NH(A)\cap \mathbb Z^n) \setminus \mathcal E_N(A)
\]
where
\[
\mathcal E_N(A) :=NH(A)\cap\left( \mathcal E(A)\cup \bigcup_{a\in A} (aN-\mathcal E(a-A) ) \right).
\]
In Theorem \ref{Thm2} we will show that this  is an equality for large $N$. We use  two classical lemmas to prove this, and include their short proofs.

 \subsection{Two classical lemmas}

\begin{lemma} [Carath\'eodory's theorem] \label{Lem:Cat} 
Assume that $0\in A$ and $A-A$ spans $\mathbb R^n$.
If $v\in NH(A)$ then there exists a subset $B\subset A$ which contains $n+1$ elements, such that $B-B$ is a spanning set for $\mathbb R^n$,  
for which $v\in NH(B)$.
\end{lemma}
 
Note that the condition $B-B$ spans $\mathbb{R}^n$ is equivalent to the condition that $B$ is not contained in any hyperplane. In two dimensions, Lemma~\ref{Lem:Cat} asserts that each point of a polygon lies in a triangle (which depends on that point) formed by 3 of the vertices. 

\begin{proof}
Since  $v \in NH(A)$ we can write
$$v=\sum_{a\in A} c_aa\in NH(A), \text{ with } 0 \leq \sum_{a \in A} c_a \leq N,$$
where each $c_a \geq 0$. We select the representation that  minimizes $\# B$ where 
$$B=\{ a: c_a>0\},$$
 Select any $b_0\in B$. We now show that the vectors $b-b_0,\ b\in B, b\ne b_0$ are linearly independent over $\mathbb R$.
If not we can write
$$\sum_{b\in B\setminus \{ b_0\} } e_b(b-b_0)=0,$$
 where the $e_b$ are not all $0$.  Let $e_{b_0}=-\sum_b e_b$ so that  $\sum_{b\in B} e_bb=0$ and $\sum_{b\in B} e_b=0$, and at least one $e_b$ is positive. Now let 
 $$m=\min_{b:\ e_b>0} c_b/e_b,$$
  where $c_\beta=me_\beta$ with $\beta\in B$. Then $v=\sum_{b\in B} (c_b-me_b)b$ where  each $c_b-me_b\geq 0$ with  $\sum_{b\in B} (c_b-me_b)=\sum_{b\in B} c_b-m\sum_{b\in B} e_b=\sum_{b\in B} c_b\in [0,N]$.
However the coefficient $c_\beta-me_\beta=0$ and this contradicts the minimality of $\#B$ .

Since the  vectors $b-b_0,\ b\in B, b\ne b_0$ are linearly independent, we can add new elements of $A$ to the set $B$ until we have $n+1$ elements, and then we obtain the result claimed.
 \end{proof}

For $u=(u_1,\dots,u_n), v=(v_1,\dots,v_n)\in \mathbb Z_{\geq 0}^n$, we write $u\leq v$ if $u_i\leq v_i$ for each $i=1,\ldots,n$. The following is a classical lemma in additive combinatorics:\footnote{Formerly known as ``additive number theory''.}
 
\begin{lemma}[Mann's lemma] \label{lem: Mann}
Let $S\subset \mathbb Z_{\geq 0}^n$. There is a finite subset $T\subset S$ such that for all $s\in S$ there exists $t\in T$ for which $t\leq s$.
\end{lemma}

\begin{proof} We prove by induction on $n\geq 1$. For convenience we will write $T\leq S$, if for all $s\in S$ there exists $t\in T$ for which $t\leq s$. For $n=1$ let $T=\{ t\}$ where $t$ is the smallest integer in $S$. For $n>1$, select 
any element $(s_1,\dots, s_n)\in S$.  Define $S_{j,r}:=\{ (u_1,\dots,u_n)\in S:\ u_j=r\}$ for each $j=1,\cdots,n$ and $0\leq r<s_j$. Let $\phi_j((u_1,\dots,u_n))=(u_1,\cdots,u_{j-1},u_{j+1},\cdots,u_n)$. The set 
$\phi_j(S_{j,r}) \subset \mathbb Z_{\geq 0}^{n-1}$ and so, by the induction hypothesis, there exists a finite subset $T_{j,r}\subset S_{j,r}$ such that $\phi_j(T_{j,r})\leq \phi_j(S_{j,r})$, which implies that $T_{j,r} \leq  S_{j,r}$ as their $j$th co-ordinates are the same. Now let
\[
 T = \{ (s_1,\dots, s_n)\}   \bigcup_{j=1}^n \bigcup_{r=0}^{s_j-1} T_{j,r},
 \]
 which is a finite union of finite sets, and so finite.
If $s\in S$ then either $(s_1,\dots, s_n)\leq s$, or $s\in S_{j,r}$ for some $j, 1\leq j\leq n$, and some
$r, 0\leq r<s_j$. Hence $T\leq S$.
\end{proof}

 \begin{lemma}[Mann's lemma, revisited] \label{lem: Mann2}
Let $S\subset \mathbb Z_{\geq 0}^n$ with the property that if $s\in S$ then $s+\mathbb Z_{\geq 0}^n\in S$.
Then $E:=\mathbb Z_{\geq 0}^n\setminus S$ is a finite union of sets of the form: For some $I\subset \{ 1,\dots, n\}$
\[
\{ (x_1,\dots,x_n): x_i\in \mathbb Z_{\geq 0} \text{ for each } i\in I\} \text{ with } x_j \text{ fixed if } j\not\in I. 
\]
\end{lemma}

\begin{proof} By induction on $n\geq 1$. In 1-dimension, $S$ is either empty so that $E=\mathbb Z_{\geq 0}$, or $S$ has some minimum element $s$, in which case $E$ is the finite set of elements $0,1,\dots,s-1$.

If $n>1$ then in $n$-dimensions either $S$ is empty so that $E=\mathbb Z_{\geq 0}^n$ or $S$ contains some element $(s_1,\dots,s_n)$.  Therefore if $(x_1,\dots,x_n)\in E$ there must exist some $k$ with $x_k\in \{ 0,1,\dots,s_k-1\}$.
For each such $k,x_k$ we apply the result to $S_{x_k}:=\{ (u_1,\dots,u_n)\in S:\ u_k=x_k\}$, which is $n-1$ dimensional.
\end{proof}

\subsection{The proof of Theorem \ref{Thm2} }
For any $v\in \mathcal P(A)$ define 
\[
\mu_A(v) := \min \left\{ \sum_{a\in A} n_a : v= \sum_{a\in A} n_aa, \text{ each } n_a \in \mathbb N \right\} ,
\]
and $\mu_A(V):=\max_{v\in V} \mu_A(v)$ for any $V \subset \mathcal P(A)$. By definition, 
$V\subset NA$ if and only if $N\geq \mu_A(V)$.

The heart of the proof of Theorem \ref{Thm2} is contained in the following result.

\begin{prop} \label{Prop3} 
Let $0\in B\subset A \subset \mathbb{Z}^n$ where $\Lambda_A=\mathbb Z^n$, and  $B^*=B\setminus \{ 0\}$ contains exactly $n$ elements, which span $\mathbb R^n$ (as a vector space over $\mathbb R$).  There exists a finite subset $A^+\subset \mathcal P(A) $ such that if $v\in \mathcal P(A) $ then there is some $w=w(v)\in A^+$ for which
$v-w\in  \mathcal P(B) $. (That is, $\mathcal P(A)=A^++\mathcal P(B)$.)
Let $N_{A,B}=\mu_A(A^+)$ so that $A^+\subset N_{A,B}A$.
If $N\geq N_{A,B}$ and 
$v\in (N-N_{A,B})H(B)\cap \mathbb Z^n$ but $v\not\in  \mathcal E(A) $ then $v\in NA$.
\end{prop}

 \begin{proof}  The  fundamental domain for the lattice $\Lambda_B:=\langle B \rangle_{\mathbb{Z}}$ is 
\[
\mathbb{R}^n / \Lambda_B \cong \mathcal F(B):=\left\{ \sum_{b\in B^*} c_bb: \text{ Each } c_b\in [0,1)\right\} .
\] 
 Since $\mathcal F(B)$ is bounded, we see that 
$$L:= \mathcal F(B)\cap  \mathbb{Z}^n$$ is finite. 
The sets $\ell + \Lambda_B$ partition $\mathbb{Z}^n$ as $\ell$ varies over $\ell\in L$.
For each $\ell\in L$ we define 
\[  A_\ell = (\ell + \Lambda_B) \cap \mathcal P(A) , \]
which partition $\mathcal P(A)$ into disjoint sets, 
so that $\mathcal P(A) =\bigcup _{\ell \in L} A_\ell$. Define $S_\ell\subset  \mathbb N^n$ by
\[
A_\ell :=\left\{ \ell+\sum_{b \in B^*} c_bb: (c_1,\ldots,c_n)\in S_\ell\right\} \subset C_B.
\]
By Mann's lemma (Lemma \ref{lem: Mann}), there is a finite subset $T_\ell\subset S_\ell$ such that for each $s \in S_\ell$ there is a $t \in T_\ell$ satisfying $t \leq s$. We may assume that $T_\ell$ is minimal, and define 
$$
A_\ell^+ =\left\{ \ell+\sum_{b\in B^*} c_bb: (c_1,\ldots,c_n)\in T_\ell\right\} \subset A_\ell.
$$
By definition, for any $v\in A_\ell$ there exists $w\in A_\ell^+$ such that $v-w\in \mathcal P(B) $ (for we write $v=\ell +s\cdot B$ and let $w=\ell+t\cdot B$ where $t\leq s$, as above). That is, $A_\ell=A_\ell^++\mathcal P(B)$.

Let $A^+=\cup_{\ell \in L} A_\ell^+$ which is a finite union of finite sets, and so is finite, and
$A^+\subset \mathcal P(A) $.  Moreover $\mathcal P(A) =\bigcup _{\ell \in L} A_\ell=\bigcup _{\ell \in L} A_\ell^++\mathcal P(B) =A^++\mathcal P(B)$ as claimed.

Now suppose that  $v\in (N-N_{A,B})H(B)\subset C_B\subset C_A$. Since the vectors in $B$ are linearly independent there is a unique representation  $v=\sum_b v_bb$ as a linear combination of the elements of $B$, and has each $v_b\geq 0$ with $\sum_b v_b\leq N-N_{A,B}$.

Also suppose $v\in \mathbb Z^n$ but $v\not\in  \mathcal E(A) $ so that $v\in   \mathcal P(A)$, as $v\in C_A\cap \mathbb Z^n$. Therefore there exists a unique $\ell\in L$ for which $v\in A_\ell$, and  $w=w(v)=\sum_b w_bb\in A_\ell^+$  for which each $0\leq w_b\leq v_b$.  Therefore $v-w=\sum_b (v_b-w_b)b\in UB$ where $U:=\sum_b (v_b-w_b)\leq \sum_b v_b\leq N-N_{A,B}$ and so $v-w\in (N-N_{A,B})B$.
By definition,  $w \in N_{A,B}A$, and so 
\[
v=(v-w)+w \in (N-N_{A,B})B+N_{A,B}A  \subset (N-N_{A,B})A +N_{A,B}A =NA.\qedhere
\]
 \end{proof}

\begin{proof}[Proof of Theorem \ref{Thm2}] 
For every subset $B\subset A$ which contains $n+1$ elements, such that $B-B$ is a spanning set for $\mathbb R^n$, define $N_{A,B}^*:=N_{A,B}+\sum_{b\in B,\ b\ne 0} N_{b-A,b-B}$, and let
$N_A$ be the maximum of these  $N_{A,B}^*$.  If $N\geq N_{A}$ and
$v\in NH(A)$ then $v\in NH(B)$ for some such set $B$, by Lemma \ref{Lem:Cat}.
If we also have $v\in  \mathbb Z^n$ but 
 \[
 v\not\in  \mathcal E(A)\cup \bigcup_{b\in B, b\ne 0}  (Nb -  \mathcal E(b-A))
 \]
then we can write
 $v=\sum_{b\in B} c_b b$ for real $c_b\geq 0$ with 
 \[
 \sum_{b\in B} c_b=N\geq N_{A,B}+\sum_{b\in B,\ b\ne 0} N_{b-A,b-B}.
 \] 
 Therefore
 
 $\bullet$\ Either  $c_0\geq N_{A,B}$ in which case
 \[
v = \sum_{b\in B, b\ne 0} c_b b \in (N-c_0)H(B) \subset (N-N_{A,B}) H(B)
 \]
 as well as $v\in \mathbb Z^n \setminus \mathcal E(A)=\mathcal P(A)$, and so $v\in NA$ by Proposition \ref{Prop3};
 
 $\bullet$\  Or there exists $\beta\in B,\ \beta\ne 0$ for which $c_\beta\geq N_{\beta-A,\beta-B}$ so that
 \[
 \beta N-v = \sum_{b\in B} c_b (\beta-b) \in (N-c_{\beta}) H(\beta-B)\subset (N-N_{\beta-A,\beta-B}) H(\beta-B).
 \]
 Now $v,\beta  \in \mathbb Z^n$ and so $\beta N-v\in \mathbb Z^n$.
 Also $v\not\in \beta N -  \mathcal E(\beta-A)$ by hypothesis, and so 
 $\beta N -v\not\in   \mathcal E(\beta-A)$. Therefore  $\beta N -v\in N(\beta-A)$ by Proposition \ref{Prop3}, giving that $v\in NA$.
 \end{proof}
 
 \section{The structure and size of the exceptional set} \label{sec:SizeSet}
 
 \begin{prop} \label{Prop5} 
Let $0\in B\subset A \subset \mathbb{Z}^n$ where $\Lambda_A=\mathbb Z^n$, and  $B^*=B\setminus \{ 0\}$ contains exactly $n$ elements, which span $\mathbb R^n$, so that $C_B=\{ \sum_{b\in B^*} x_bb:\text{ Each } x_b\geq 0\}$.
There exist $r_b\geq 0$ such that $\{ \sum_{b\in B^*} x_bb:\text{ Each } x_b\geq r_b\}\cap \mathbb{Z}^n \subset \mathcal P(A)$.
\end{prop}

We deduce that if $x:=\sum_{b\in B^*} x_bb\in (C_B\cap \mathbb{Z}^n)\cap \mathcal E(A)$ then  $0\leq x_b<r_b$ for some $b$.
In other words $x$ is at a bounded distance from the boundary generated by $B\setminus \{ b\}$.
(Theorem 2 of  \cite{ST} gives a related result but is difficult to interpret  in the language used here.)

 \begin{proof} We will use the notation of Proposition \ref{Prop3}.
  The elements of $B^*$ are linearly independent so that $\beta:=\sum_{b\in B} b$ lies in the interior of 
 $C_B$. Therefore if the integer $M$ is sufficiently large then $\gamma:=\beta +\tfrac 1M \sum_{a\in A} a$ also lies in the interior of 
 $C_B$.
 
Now as $A$ generates $\mathbb Z^n$ as a vector space over $\mathbb Z$, we know that for each $\ell\in L$ there exist integers
$c_{\ell, a}$ such that $\ell= \sum_{a\in A} c_{\ell, a}a$. Let $c\geq 0$ be an integer $\geq \max_{\ell\in L, a\in A} (-c_{\ell, a})$. The set $L'=cM \gamma + L$ of $\mathbb Z^n$-points  is a translate of $L$ that can be represented as 
\[
cM \gamma +  \sum_{a\in A} c_{\ell, a}a = cM \beta  + \sum_{a\in A} (c+c_{\ell, a})a \in  \mathcal P(A) \text{ for each } \ell \in L.
\]
The translation is by $cM \gamma\in C_B$ so $L'=cM \gamma + L\subset C_B$; moreover $L'$ gives a complete set of representatives
of $\mathbb{R}^n / \Lambda_B$ and so every lattice point in $L'+\mathcal P(B)$ belongs to $\mathcal P(A)$.
We can re-phrase this as
\[
(cM \gamma+ C_B) \cap \mathbb{Z}^n \subset \mathcal P(A).
\]
Therefore if $cM \gamma = \sum_{b\in B^*} r_bb$ and $x:=\sum_{b\in B^*} x_bb\in \mathbb{Z}^n$, then 
$x\in \mathcal P(A)$ if each $x_b\geq r_b$.   
\end{proof}

\begin{proof}[Proof of Theorem \ref{Thm3}] 
We again use Lemma~\ref{Lem:Cat} to focus on sets $B\subset A$ which contain $n+1$ elements, such that $B-B$ is a spanning set for $\mathbb R^n$. We translate $B$ so that $0\in B$. As in the proof of Proposition \ref{Prop3}, we fix $\ell \in L$ (which is a finite set). Proposition \ref{Prop5} shows that $S_\ell$ is non-empty.
Lemma \ref{lem: Mann2} yields  the structure of   $ \mathbb Z_{\geq 0}^n\setminus S_\ell$, which is not all of $ \mathbb Z_{\geq 0}^n$ as $S_\ell$ contains an element. This implies that the structure of
$(\ell + \Lambda_B) \cap \mathcal E(A)$ is as claimed in Theorem \ref{Thm3}. The result follows as $\mathcal E(A)$ is a finite union of such sets.
\end{proof}



\end{document}